\documentclass[12pt,reqno]{amsart}
\usepackage{amsmath, amsthm, amssymb, stmaryrd}
\usepackage{hyperref}
\usepackage{enumerate}
\usepackage{url}
\usepackage{color}
\usepackage{comment}

\let\OLDthebibliography\thebibliography
\renewcommand\thebibliography[1]{
  \OLDthebibliography{#1}
  \setlength{\parskip}{0pt}
  \setlength{\itemsep}{0pt plus 0.3ex}
}

\usepackage{mathtools}

\usepackage{multirow, array}
\usepackage{placeins}

\usepackage{caption} 
\advance \topmargin by -\headheight
\advance \topmargin by -\headsep
     
\evensidemargin \oddsidemargin
\newtheorem{theorem}{\bf Theorem}[section]

\newtheorem{lemma}[theorem]{\bf Lemma}
\newtheorem{corollary}[theorem]{\bf Corollary}

\newtheorem{proposition}[theorem]{\bf Proposition}

\numberwithin{equation}{section}

\newcommand*\wrapletters[1]{\wr@pletters#1\@nil}
\def\wr@pletters#1#2\@nil{#1\allowbreak\if&#2&\else\wr@pletters#2\@nil\fi}

\usepackage{fixme}
\fxsetup{
    status=draft,
    author=,
    layout=margin,
    theme=color
}

\begin{document}
\title[4-AP free sets with 3-APs in all large subsets]{Four-term progression free sets with three-term progressions in all large subsets}

\author{Cosmin Pohoata}
\address{Yale University, New Haven, US 06511.}
\email{andrei.pohoata@yale.edu}

\author{Oliver Roche-Newton}
\address{Johann Radon Institute for Computational and Applied Mathematics (RICAM), Linz, Austria}
\email{o.rochenewton@gmail.com}


\thanks{}
\date{}
\begin{abstract} 
This paper is mainly concerned with sets which do not contain four-term arithmetic progressions, but are still very rich in three term arithmetic progressions, in the sense that all sufficiently large subsets contain at least one such progression. We prove that there exists a positive constant $c$ and a set $A \subset \mathbb F_q^n$ which does not contain a four-term arithmetic progression, with the property that for every subset $A' \subset A$ with $|A'| \geq |A|^{1-c}$, $A'$ contains a nontrivial three term arithmetic progression. We derive this from a more general quantitative Roth-type theorem in random subsets of $\mathbb{F}_{q}^{n}$, which improves a result of Kohayakawa-Luczak-R\"odl/Tao-Vu.

We also discuss a similar phenomenon over the integers, where we show that for all $\epsilon >0$, and all sufficiently large $N \in \mathbb N$, there exists a four-term progression-free set $A$ of size $N$ with the property that for every subset $A' \subset A$ with $|A'| \gg \frac{1}{(\log N)^{1+\epsilon}} \cdot N$ contains a nontrivial three term arithmetic progression.

Finally, we include another application of our methods, showing that for sets in $\mathbb{F}_{q}^{n}$ or $\mathbb{Z}$ the property of ``having nontrivial three-term progressions in all large subsets" is almost entirely uncorrelated with the property of ``having large additive energy".
\end{abstract}
\maketitle

\section{Introduction}
\label{intro}


A $k$-term arithmetic progression in an additive group is a set of the form $\{x,x+d, \dots, x+(k-1)d\}$. If $d \neq 0$, then we say that the progression is non-trivial. The shorthand $k$-AP is used for a $k$-term arithmetic progression. If a set $A$ does not contain any non-trivial $k$-APs, we say that $A$ is \textit{$k$-AP free}.

We define $f_k(A)$ to be the size of the largest $k$-AP free subset of $A$. In the case when $A=\{1,\dots,N\}\subset \mathbb Z$, the study of the behaviour of $f_k(A)$ has been a central topic in additive combinatorics. Following the standard notation, we will write
\[
r_k(N):=f_k(\{1,\dots,N\}).
\]
The seminal result on this topic is Szemer\'{e}di's Theorem \cite{S}, which states that sets of integers with positive density contain arbitrarily long arithmetic progressions, or using the notation above $r_k(N)=o(N)$. Szemer\'{e}di's Theorem generalized Roth's Theorem \cite{R}, which had earlier established the case when $k=3$. There has since been a great deal of research aimed at finding the correct asymptotic behaviour of $r_k(N)$, particularly in the case when $k=3$. The current state-of-the-art is that
\begin{equation} \label{r3}
\frac{\log^{1/4}{N}}{2^{2\sqrt{2} \sqrt{\log N}}} \cdot N \ll r_3(N) \ll \frac{1}{(\log N)^{1+c}} \cdot N,
\end{equation}
for some absolute constant $c>0$. The upper bound in \eqref{r3} is due to a recent breakthrough result of Bloom and Sisask \cite{BlSi}, while the lower bound comes from Elkin \cite{ME}, who improved upon the celebrated construction of Behrend \cite{Be}. For more background and history on the behaviour of $r_3(N)$, see \cite{Bl}, \cite{BlSi}, and the references within.

Similar problems have been studied in other settings, and of particular relevance to this work is the setting of $\mathbb F_q^n$. A recent breakthrough of Croot, Lev and Pach \cite{CLP} and Ellenberg and Gijswijt \cite{EG} gave spectacular quantitative progress over $\mathbb F_q^n$, resulting in the bound
\begin{equation} \label{eq:EG}
f_3(\mathbb F_q^n) \leq q^{n(1-c_q)}
\end{equation}
where $c_q>0$ can be calculated explicitly (and satisfies $c_{q} = \Theta((\log q)^{-1})$ as $q$ grows large); see the forthcoming Section \ref{sec:ss} for more details). Note that this bound is much better than what one could hope to prove for the corresponding problem over the integers, which highlights that this change of setting leads to a rather different problem.

In this paper, we consider a problem in this direction but with a slightly different flavour. Let $k \geq 3$ be an integer and suppose that we have a set $A$ in a group $G$ which does not contain any $(k+1)$-APs. Is it always possible to find a large subset of $A$ which does not contain any $k$-APs? Or using the notation we have established, is it always the case that $f_k(A)$ is large when $A$ is $(k+1)$-AP free? 

Perhaps a first intuitive guess is that the answer should be ``yes", and that all $k$-APs can be destroyed by deleting a relatively small number of elements of $A$. Focusing on the situation when $k=3$ and $G$ is $ \mathbb Z$ or $\mathbb F_q^n$, the results of this paper give quantitative answers to this question in the negative direction. 

Define the constant $C_q$ to be
$$C_{q} = 1 + \frac{1}{c_{q}},$$
where $c_q$ is the constant mentioned above in \eqref{eq:EG}. Our main result is the following.

\bigskip

\begin{theorem} \label{thm:main}
    For all $\beta >0$, there exists $n_0=n_0(\beta)$ such that the following statement holds for all $n \geq n_0$ and for any prime power $q$. There exists a four-term progression free set $A \subset \mathbb F_q^n$ such that
    \[
    f_3(A) \leq |A|^{1-\frac{1}{2(C_q-2)}+\beta}.
    \]
\end{theorem}

\bigskip

That is, we show the existence of a set $A \subset \mathbb F_q^n$ which does not contain a non-trivial 4-AP but for which every large subset $A' \subset A$ contains a non-trivial 3-AP.
For a concrete example, one can calculate that $C_5 \approx 15.12589$, meaning that every set $A'\subset A$ larger than $|A|^{0.962}$ contains a 3-AP. 

Our proof relies on an iterated application of the so-called hypergraph container theorem, which we will describe in the next section, and which takes as input a supersaturated version of the subexponential Ellenberg-Gijswijt upper bound for 3-AP free subsets of $\mathbb F_{q}^{n}$ from \eqref{eq:EG}. In fact, we will derive Theorem \ref{thm:main} from a more general result about random subsets of $\mathbb{F}_{q}^{n}$, in the spirit of Kohayakawa-\L uczak-R\"{o}dl \cite{KLR} and Conlon-Gowers \cite{CG}.

\begin{theorem} \label{thm:CG'}
Let $\beta > 0$, $t < c_q(1-2\beta)$ and let $p$ be a positive real number satisfying
$$q^{n\left( -\frac{1}{2}+\frac{t(C_q-1)}{2}\right)} \leq p \leq 1.$$
Let $B$ be a random subset of $\mathbb{F}_{q}^{n}$ with the events $x \in B$ being independent with probability $\mathbb{P}(x \in B) = p$. Then, with probability $1 - o_{n \to \infty}(1)$ we have that
$$f_{3}(B) \ll pq^{n\left(1-t +2\beta\right)}.$$
In particular, for all $\epsilon>0$, there exists $\delta(\epsilon,q):=\delta >0$ such that if $B$ is defined as above with $p=q^{n(-\frac{1}{2}+\epsilon)}$, then with probability $1 - o_{n \to \infty}(1)$,
\[
f_3(B) \ll |B|^{1-\delta}.
\]
\end{theorem}

This allows us to detect three-term arithmetic progressions in subsets of $\mathbb{F}_{q}^{n}$ of size as small as $q^{n\left( \frac{1}{2}+\epsilon \right )}$, which is beyond the reach of the Ellenberg-Gijswijt bound \eqref{eq:EG}, provided that those subsets have large relative density compared to a random set. This improves a result of Tao and Vu from \cite[Theorem 10.18]{tv}. It is also worth pointing out that the range for $p$ in Theorem \ref{thm:CG'} is optimal. Indeed, if $p = q^{-n/2}/2$, then the expected number of three-term progressions in a random subset $B$ of $\mathbb{F}_{q}^{n}$ (where each element in $B$ is chosen independently with probability $p$) is less than $q^{n/2}/8$, while the expected number of elements in $B$ is $q^{n/2}/2$. Therefore, one can almost always remove an element from each progression and still be left with at least half the elements of $B$.

We also consider the analogue of Theorem \ref{thm:main} in the integer setting, where we obtain the following result. 

\smallskip

\begin{theorem} \label{thm:integers}
There exists $c'>0$ such that for all $N \in \mathbb N$ sufficiently large, there exists a set of integers $A$ with $|A|=N$ which does not contain any nontrivial four-term arithmetic progression, and for which
\begin{equation} \label{intbound}
f_3(A) \ll \frac{1}{(\log N)^{1+c'}} \cdot N.
\end{equation}
\end{theorem}

\smallskip

The constant $c'$ above is dependent on the constant $c$ in the upper bound on $r_3(N)$ in \eqref{r3}. 

It is important to mention that in the integer setting, if merely a sublinear upper bound on $f_{3}(A)$ would be the goal, one could could pretty easily explicitly describe a set of integers $A$ with no four-term progressions for which the powerful density Hales-Jewett theorem ensures that all of its relatively dense subsets share the same property; consider, for instance, the subset of the first $N$ integers with only digits $0$, $1$ or $2$ in base $6$. This is a 4-AP-free set $A$ for which indeed $f_{3}(A)=o(|A|)$ but the asymptotic notation doesn't hide good bounds. In the (non-quantitative) direction, a much more general statement was also recently established by Balogh, Liu and Sharifzadeh in \cite[Theorem 1.7]{BLS}, who show that for all $k \geq 3$, there exists a set $S$ of primes such that $S$ is $(k+1)$-AP free, and $f_k(S)=o(|S|)$. Theorem \ref{thm:integers} should perhaps be thought of as follows: there exist sets of $N$ integers without non-trivial four-term progressions for which the size of the largest 3-AP free subset is smaller than roughly the best upper bound known for $r_{3}(N)$.

After discussing the required ingredients in Sections 2 and 3, we prove Theorems \ref{thm:main} and \ref{thm:integers} in Sections 4 and 5, respectively. In Section 6, we will discuss another application of our methods, showing that for sets (in $\mathbb{F}_{q}^{n}$ or $\mathbb{Z}$) the property of ``having nontrivial three-term progressions in all large subsets" is almost entirely uncorrelated with the property of ``having large additive energy". In particular, we prove the existence of sets $A$ with minimal additive energy and small $f_3(A)$.

\smallskip

\subsection*{Asymptotic Notation}

Throughout the paper, the standard notation
$\ll,\gg$ and respectively $O$ and $\Omega$ is applied to positive quantities in the usual way. That is, $X\gg Y$, $Y \ll X,$ $X=\Omega(Y)$ and $Y=O(X)$ all mean that $X\geq cY$, for some absolute constant $c>0$. If both $X \ll Y$ and $Y \ll X$ hold, we write $X \approx Y$, or equivalently $X= \Theta(Y)$. If the constant $c$ depends on a quantity $k$, we write $X \gg_k Y$, $Y=\Omega_k(Y)$, and so on.


\smallskip

\subsection*{Funding and acknowledgments}

ORN was partially supported by the Austrian Science Fund FWF Project P 30405-N32. We are grateful to Tom Bloom, Christoph Koutschan, Fernando Shao, Maryam Sharifzadeh and Adam Zsolt-Wagner for helpful conversations and advice.

\bigskip

\section{The Container Theorem}

A critical tool in this paper comes from the theory of hypergraph containers. The statement that we use is rather technical, but it can be roughly summarised as follows: if a hypergraph $H=(V,E)$ has a good edge distribution (in the sense that no vertices have unusually large degree, and more generally the elements of any set of vertices do not share too many common edges), then we obtain strong information about the independent sets of the hypergraph. This strong information is that there is a family $\mathcal C$ of subsets of $V$ such that
\begin{itemize}
    \item for every independent set $X \subset V$, there is some $A \in \mathcal C$ such that $X \subset A$,
    \item $\mathcal C $ is not too large,
    \item each $A \in \mathcal C$ does not have too many edges.
\end{itemize}

The theory of hypergraph containers was developed independently by Balogh, Morris and Samotij \cite{BMS} and Saxton and Thomasson \cite{ST}. For a recent survey on this topic, see \cite{BMS2}. This method has led to several significant breakthroughs in combinatorics in recent years, most notably in the field of extremal graph theory. However, this purely combinatorial tool has also led to new results in additive combinatorics. For example, it was proven by Balogh, Liu and Sharifzadeh \cite{BLS} that, for infinitely many $N \in \mathbb N$ there are at most $2^{O(r_k(N))}$ subsets of $[N]$ which do not contain a $k$-AP. Note that this is almost best possible, since any subset of a $k$-AP free set is $k$-AP free, and so the subsets of the largest $k$-AP free set give at least $2^{r_k(N)}$ sets which are $k$-AP free. Another application of containers closely related to (and which inspired) this paper can be found in Balogh and Solymosi \cite{BS}, where it was proven that there exists a set $P$ of $N$ points in the plane such that $P$ does not contain any collinear quadruples, but any subset of $P$ of size larger than $N^{5/6+o(1)}$ contains a collinear triple. The arguments used to prove the main results in this paper follow a similar structure to those of \cite{BS}.

In order to state the required hypergraph container result formally, we need to introduce some more notation. Let $H=(V,E)$ be an $r$-uniform hypergraph. Write $e(H)=|E|$. For any $S \subset V$, the subhypergraph induced by $S$ is denoted by $H[S]$.
The co-degree of $S$ is the quantity
\[
d(S):= | \{ e \in E : S \subseteq e \}|.
\]
In the case when $S=\{v\}$ is a singleton, we simply write $d(v)$. The average degree of a vertex in $H$ is denoted by $d$, that is,
\[
d=\frac{1}{|V|}\sum_{v \in V} d(v)=\frac{r|E|}{|V|}.
\]
For each $2 \leq j \leq r$, denote
\[
\Delta_j(H):= \max_{S \subset V : |S|=j} d(S).
\]
For $0<\tau<1$, define the function
\[
\Delta(H, \tau)= 2^{ \binom{r}{2}  - 1} \sum_{j=2}^r \frac{\Delta_j(H)}{2^{\binom{j-1}{2}} d\tau^{j-1}}.
\]
This function gives a measure of how well-distributed the edges of $H$ are. In this paper, we will only consider $3$-uniform hypergraphs, in which case the function can be expressed more straightforwardly:
\[
\Delta(H, \tau) = \frac{4\Delta_2(H)}{d \tau} + \frac{2\Delta_3(H)}{d \tau^2}.
\]

The exact result that we will use is Corollary 3.6 in \cite{ST}. 

\begin{theorem} \label{thm:container}
Let $H=(V,E)$ be and $r$-uniform hypergraph with $|V|=N$. Let $0< \epsilon, \tau < 1/2$ satisfy the conditions that
\begin{itemize}
    \item $\tau < 1/(200 \cdot r \cdot r!^2)$,
    \item $\Delta(H, \tau) \leq \frac{\epsilon}{12r!}$.
\end{itemize}
Then there exists $c=c(r) \leq 1000 \cdot r \cdot r!^3$ and a collection $\mathcal C$ of subsets of $V(H)$ such that 
\begin{itemize}
    \item if $X \subseteq V$ is an independent set, then there is some $A \in \mathcal C$ such that $X \subseteq A$,
    \item for every $A \in \mathcal C$, $e(H[A]) \leq  \epsilon e(H)$, 
    \item $\log | \mathcal C| \leq cN\tau \cdot \log(1/\epsilon) \cdot \log(1 / \tau)$.
\end{itemize}

\end{theorem}

\bigskip

\section{Supersaturation Results} \label{sec:ss}

In most applications of the container method, a crucial ingredient is a so-called Supersaturation Lemma. Extremal results in combinatorics often state that sufficiently large subsets of a given set contain at least one copy of some special structure. A supersaturation result goes further, and says that sufficiently dense subsets of a given set contain \textit{many} copies of certain structures.

In our particular setting we can be more concrete. We need to prove that sufficiently large subsets of $\mathbb F_q^n$ and $[N]$ contain many 3-APs. The results and techniques in these two different settings differ significantly, particularly in light of recent developments concerning the size of the largest 3AP-free set in $\mathbb F_q^n$ in \cite{CLP} and \cite{EG}.

\subsection{Supersaturation in $\mathbb F_q^n$} 


We begin by finally defining the previously mentioned constant $c_q$ by 
\[
q^{1-c_q} = \inf_{0 < y < 1} \frac {1+y+ \dots +y^{q-1}}{y^{(q-1)/3}}.
\]
Also, recall that $C_q:=1+\frac{1}{c_q}$. For a fixed $q$, these constants $c_q$ and $C_q$ can be calculated explicitly. 

Define a \textit{triangle} in $\mathbb F_q^n$ to be a triple $(x,y,z) \in \mathbb F_q^n \times \mathbb F_q^n \times \mathbb F_q^n$ such that $x+y+z=0$. To obtain a supersaturation result for arithmetic progressions in $\mathbb F_q^n$, we will make use of the following result of Fox and Lov\'{a}sz \cite{FL}.

\smallskip

\begin{theorem} \label{thm:FL}
Let $0 < \epsilon < 1$ and $\delta= (\epsilon/3)^{C_q}$. If $X,Y,Z \subset \mathbb F_q^n$ with less than $\delta q^{2n}$
triangles in $X \times Y \times Z$,
then we can remove $\epsilon q^n$ elements from $X \cup Y \cup Z$ so that no triangle remains.
\end{theorem}

This implies the following corollary.

\smallskip

\begin{corollary} \label{thm:supersaturation}
Let $A \subset \mathbb F_q^n$ with $|A| = q^{n(1-s)}$, $0 \leq s < c_q$ and suppose that $n$ is sufficiently large. Then $A$ contains $\Omega_q(q^{n(2-sC_q)})$ non-trivial three term arithmetic progressions.
\end{corollary}

\smallskip

\begin{proof} Applying the bound \eqref{eq:EG}, we know that for some constant $k$, every subset of $A$ with size greater than $kq^{n(1-c_q)}$ contains a three term arithmetic progression. Let $\epsilon=\frac{1}{2q^{ns}}$. It therefore follows that, for $n$ sufficiently large,
\[
 |A| - \epsilon q^n = \frac{q^{n(1-s)}}{2} \geq kq^{n(1-c_q)}.
\]
In particular, any subset of $A$ of size $|A|-\epsilon q^n$ contains a non-trivial $3$-AP. To put it another way, if we remove $\epsilon q^n$ elements from $A$, the resulting set still contains a $3$-AP.

Now we can apply Theorem \ref{thm:FL} in its contrapositive form with $X=Y=A$ and $Z=-2A$, so that the property of being triangle free is the same as that of being $3$-AP free. It follows that $A \times A \times (-2A)$ contains at least
\[ \delta q^{2n}=\left(\frac{1}{6q^{ns}}\right )^{C_q}q^{2n} = k'(q) q^{n(2-sC_q)}
\]
triangles. Some of these triangles may correspond to trivial arithmetic progressions, but the number of such progressions is negligible and the proof is complete.
\end{proof}

\subsection{Supersaturation in the integers}

A supersaturation lemma for three term arithmetic progressions in $[N]$ is already standard, in the form of Varnavides' Theorem. We will use the following formulation, which can be derived from Lemma 3.1 in \cite{CS}.

\bigskip

\begin{theorem} \label{thm:varnavides} Suppose that for all $N \in \mathbb N$ we have $r_3(N) \leq \frac{N}{h(N)}$ for some invertible function $h:\mathbb R^+ \rightarrow \mathbb R^+$. Then for all $A \subset [N]$ with cardinality $|A| = \eta N$, such that
\[
1 \leq \left \lfloor h^{-1} \left( \frac{4}{\eta} \right) \right \rfloor \leq N
\]
$A$ contains at least 
\[
\left( \frac{\eta}{2(h^{-1}(\frac{4}{\eta}))^4} \right)N^2,
\]
non-trivial three term arithmetic progressions.
\end{theorem}

Theorem \ref{thm:varnavides} follows by applying \cite[Lemma 3.1]{CS} with $M=h^{-1}\left(\frac{4}{\eta} \right )$.

\bigskip

\section{Proof of Theorem \ref{thm:main} via Theorem \ref{thm:CG'}}

\bigskip

The proof of Theorem \ref{thm:CG'} begins by iteratively applying the container theorem to subsets of $\mathbb F_q^n$ in order to establish the existence of a convenient family of sets $\mathcal C$ which contain all 3-AP free subsets of $\mathbb F_q^n$. This results in the following container lemma.

\smallskip

\begin{lemma} \label{lem:containers}
For all $\beta>0$ and for all $0\leq t \leq c_q(1-3\beta)$ there exists a constant $c=c(q,\beta)$ such that there exists a family $\mathcal C$ of subsets of $\mathbb F_q^n$ with the following properties:
\begin{itemize}
\item $|\mathcal C| \leq 2^{n^2c(q,\beta)q^{n\left(\frac{1}{2}+\beta+\frac{t(C_q-3)}{2}\right )}}$,
    \item for all $A \in \mathcal C$, $|A| \leq q^{n\left (1- t \right)}$,
    \item if $X \subset \mathbb F_q^n$ is 3-AP free, then there exists $A \in \mathcal C$ such that $X \subseteq A$.
\end{itemize}

\end{lemma}

\smallskip

\begin{proof}

At the outset, this problem is converted into a graph theoretic situation in order to setup an application of Theorem \ref{thm:container}. Given $A \subset \mathbb F_q^n$, define a 3-uniform $H(A)=(V,E)$ hypergraph with vertex set $V=A$. Three distinct vertices form an edge in $H$ if and only if they form a three term arithmetic progression. 

The aim is to find a good set of containers for the hypergraph $H(\mathbb F_q^n)$. We will eventually obtain a family $\mathcal C$ of subsets of $\mathbb F_q^n$ such that
\begin{itemize} \label{claim}
    \item $|\mathcal C| \leq 2^{n^2c(q,\beta)q^{n\left(\frac{1}{2}+\beta+\frac{t(C_q-3)}{2}\right )}}$,
    \item for all $A \in \mathcal C$, $|A| \leq q^{n\left (1-t \right)}$,
    \item if $X$ is an independent set in the hypergraph $H(\mathbb F_q^n)$, then there is some $A \in \mathcal C$ such that $X \subseteq A$.
\end{itemize}


Once the existence of such a family $\mathcal C$ has been established, the proof of Lemma \ref{lem:containers} will be complete.

We will iteratively apply the container theorem to subsets of $\mathbb F_q^n$. We begin by applying Theorem \ref{thm:container} to the graph $H(\mathbb F_q^n)$. As a result, we obtain a set $\mathcal C_1$ of containers. We iterate by considering each $A \in \mathcal C_1$. If $A$ is not small enough, then we apply Theorem \ref{thm:container} to the graph $H(A)$ to get a family of containers $\mathcal C_A$. If $A$ is sufficiently small, then we put this $A$ into a final set $\mathcal C$ of containers (or to put it another way, we write $\mathcal C_A=A$). 

Repeating this for all $A \in \mathcal C_1$ we obtain a new set of containers
\[
\mathcal C_2 = \bigcup_{A \in  \mathcal C_1} \mathcal C_A.
\]
Note that $\mathcal C_2$ is a container set for $H(\mathbb F_q^n)$. Indeed, suppose that $X$ is an independent set in $H(\mathbb F_q^n)$. Then there is some $A \in \mathcal C_1$ such that $X \subset A$. Also, $X$ is an independent set in the hypergraph $H(A)$, which implies that $X \subset A'$ for some $A' \in \mathcal C_A \subset \mathcal C_2$.

We then repeat this process, defining
\[
\mathcal C_i = \bigcup_{A \in  \mathcal C_{i-1}} \mathcal C_A.
\]
By choosing the values of $\tau$ and $\epsilon$ appropriately, we can ensure that after relatively few steps we have all of the elements of $\mathcal C_k$ sufficiently small. We then declare $\mathcal C= \mathcal C_k$. It turns out that, because of $k$ being reasonably small, $|\mathcal C|$ is also fairly small.

Now we give more precise details of how to run this argument. Let $A \in \mathcal C_j$, with $j \leq k$, where $k$ is the total number of steps in the iterative process, the value of which will be specified precisely later. Write $|A|=q^{n(1-s)}$. If $s \leq t$, then apply the container theorem to $H(A)$ with
\[ \epsilon = q^{-\beta n}, \,\,\,\,\,\,\,\,\,\, \tau= q^{\frac{n}{2} (2\beta -1 +s(C_q-1))}.
\]
In order to apply the container theorem, we need to check that the conditions $\tau < 1/(200 \cdot 3 \cdot 3!^2)= 1/21600$, and $\Delta(H, \tau) \leq \frac{\epsilon}{72}$ hold. The first of these conditions will hold if we take $n$ sufficiently large. This follows from the condition that $s \leq t \leq c_q(1-3\beta)$.

For the second condition, we need to verify that
\begin{equation} \label{eq:condition}
 \frac{4 \Delta_2}{d \tau} + \frac{2 \Delta_3}{d \tau^2}  \leq \frac{\epsilon}{72}.
\end{equation}
Observe that, for any subset $A \subset  \mathbb F_q^n$, $\Delta_2(H(A))\leq 3$, since for any two distinct elements $a_1,a_2 \in A$, there are at most three possible choices of a third element $a_3 \in A$ such that $\{a_1,a_2,a_3\}$ forms an arithmetic progression. We also have $\Delta_3(H(A)) \leq 1$.

To bound the average vertex degree $d$, we use Corollary \ref{thm:supersaturation}. The set $A$ has cardinality $q^{n(1-s)}$, implying that it contains $\Omega_q(q^{n(2-sC_q)})$ non-trivial three-term arithmetic progressions. Therefore, 
\[ d = \frac{3 |E(H(A))|}{|A|} \gg_q \frac{ q^{n(2-sC_q)}}{q^{n(1-s)}}= q^{n(1-s(C_q-1))}.
\]
Therefore, it follows that, for some constant $c_0$ depending on $q$,
\[ \frac{4 \Delta_2}{d \tau} + \frac{2 \Delta_3}{d \tau^2} \leq \frac{12}{d \tau} + \frac{2}{d \tau^2} \leq \frac{14}{d\tau^2} < \frac{c_0}{q^{2\beta n}} \leq \frac{\epsilon}{72},
\]
where the last inequality holds for all $n$ sufficiently large. This verifies the condition \eqref{eq:condition}, and so we can apply Theorem \ref{thm:container} and obtain a set of containers $\mathcal C_A$ with
\[ | \mathcal C_A| \leq 2^{cq^{n(1-s)}\tau \cdot \log(1/\epsilon) \cdot \log(1 / \tau)} \leq 2^{c(n\log q)^2q^{\frac{n}{2}(1+s(C_q-3)+2\beta)}}. \]
Since $s \leq t$, it follows that we have the bound
\[ | \mathcal C_A| \leq 2^{c(n\log q)^2q^{\frac{n}{2}\left (1+t(C_q-3)+2\beta\right )}}.\]
We also know that, for each $B \in \mathcal C_A$,
\[e(H(B)) \leq \epsilon e(H(A)) = q^{-\beta n}e(H(A)).
\]
Therefore, at the $i$th level of this iterative procedure, a container $B \in \mathcal C_i$ satisfies
\[ e(H(B)) \leq q^{n(2-i\beta)}.
\]
This is good, because after $c(\beta)$ steps we can ensure that $e(H(B))$ is sufficiently small so that we can apply Theorem \ref{thm:supersaturation} and deduce that $|B| \leq q^{n\left (1-t\right)}$. In particular, if we take 
\[
k:=\left \lceil \frac{tC_q}{\beta}+1 \right \rceil ,
\]
then Theorem \ref{thm:container} tells us that for each $B \in \mathcal C_k$, $|B| \leq q^{n(1-t)}$.

So, the process terminates after at most $k$ steps. This implies that the final set of containers $\mathcal C= \mathcal C_k$ has cardinality
\[|\mathcal C| \leq 2^{c (n\log q)^2 k q^{\frac{n}{2}\left(1+t(C_q-3)+2\beta\right)}} = 2^{n^2c(q,\beta) q^{\frac{n}{2}\left(1+t(C_q-3)+2\beta\right)}} ,
\]
as claimed.

\end{proof}

The set of containers established in Lemma \ref{lem:containers} can now be used to deduce Theorem \ref{thm:CG'}, which we recall for the reader's convenience.

\bigskip

\noindent {\bf{Theorem \ref{thm:CG'}}}. Let $\beta > 0$, $t \leq c_q(1-3\beta)$ and let $p$ be a positive real number satisfying
$$q^{n\left( -\frac{1}{2}+\frac{t(C_q-1)}{2} -\frac{\beta}{2} \right)} \leq p \leq 1.$$
Let $B$ be a random subset of $\mathbb{F}_{q}^{n}$ with the events $x \in B$ being independent with probability $\mathbb{P}(x \in B) = p$. Then, with probability $1 - o_{n \to \infty}(1)$ we have that
$$f_{3}(B) \ll pq^{n\left(1-t +2\beta\right)}.$$
In particular, for all $\epsilon>0$, there exists $\delta(\epsilon,q):= \delta >0$ such that if $B$ is defined as above with $p=q^{n(-\frac{1}{2}+\epsilon)}$, then with probability $1 - o_{n \to \infty}(1)$,
\[
f_3(B) \ll |B|^{1-\delta}.
\]

\bigskip

\begin{proof}[Proof of Theorem \ref{thm:CG'}] For convenience, define $m=pq^{n\left(1-t +2\beta\right)}$, and let $\mathcal C$ be the container set guaranteed by Lemma \ref{lem:containers}. We first note that the probability that $B$ contains a three-term progression-free subset of size at least $m$ is upper bounded by
\begin{equation} \label{probbound}
 |\mathcal C| \binom{q^{n\left (1-t\right)}}{m} p^m.
\end{equation}
This is because a 3-AP free set of size $m$ must be contained in some $A \in \mathcal C$, and each subset of size $m$ belongs to the random subset $B$ with probability $p^m$. Every $A \in \mathcal C$ has size
\[
|A| \leq q^{n\left (1-t\right)},
\]
and so the number of possible candidates for a 3-AP free set of size $m$ is at most
\[
 |\mathcal C| \binom{q^{n\left (1-t\right)}}{m}.
\]
An application of the union bound then gives \eqref{probbound}. Using the bound
$$|\mathcal C| \leq  2^{n^2c(q,\beta) q^{n\left(\frac{1}{2}+\frac{t(C_q-3)}{2}+\beta\right)}},$$
and the standard binomial coefficient estimate $\binom{s}{t} \leq \left(\frac{es}{t} \right)^t$ gives
\begin{eqnarray} |\mathcal C| \binom{q^{n\left (1-t\right)}}{m} p^m &\leq& 2^{n^2c(q,\beta)q^{n\left(\frac{1}{2}+\frac{t(C_q-3)}{2}+\beta\right )}} \left( \frac{epq^{n\left (1-t \right)}}{m}\right )^m \nonumber \\
&=& 2^{n^2c(q,\beta)q^{n\left(\frac{1}{2}+\frac{t(C_q-3)}{2}+\beta\right )}} \left( \frac{e}{q^{2\beta n}}\right )^m \nonumber  \\
&\leq& \left( \frac{2e}{q^{2\beta n}}\right )^m. \label{final}
\end{eqnarray}
In the last inequality above, we have used the fact that for $n$ sufficiently large, 
\[
m=pq^{n\left(1-t +2\beta\right)}  \geq q^{n \left(\frac{1}{2}+\frac{t(C_q-3)}{2} +\frac{3}{2}\beta \right)}
\geq n^2c(q,\beta)q^{n\left(\frac{1}{2}+\frac{t(C_q-3)}{2}+\beta \right )}.
\]
The lower bound on $p$ in the statement of the theorem was used here. The quantity in \eqref{final} tends to zero as $n$ goes to infinity, which completes the proof of the first part of the statement.

The second statement follows from the first by taking 
\[ t=\frac{2\epsilon}{C_q-1},\,\,\,\, \beta=t/4.
\]
Indeed, for suitably chosen constants $c, C>0$, the statement
\[cpq^n \leq |B| \leq C pq^n = C q^{n(\frac{1}{2}+\epsilon)}
\]
is true with probability $1 - o_{n \to \infty}(1)$. Therefore, with probability $1 - o_{n \to \infty}(1)$, we have
\[
f_3(B) \ll pq^{n(1-\frac{t}{2})} \ll |B|q^{n(-\frac{t}{2})} \ll_{\epsilon} |B|^{1-\delta(\epsilon)}.
\]
\end{proof}

\smallskip

We finally use Theorem \ref{thm:CG'} to deduce Theorem \ref{thm:main}.

\smallskip

\begin{proof}[Proof of Theorem \ref{thm:main}]
Construct a subset $P \subset \mathbb F_q^n$ by choosing elements independently at random with probability $p=\frac{1}{100} q^{-n/3}$. The expected number of elements in $P$ is $pq^n=\frac{1}{100}q^{2n/3}$, while the expected number of nontrivial four-term progressions is at most $p^4q^{2n}=10^{-8}q^{2n/3}$. Indeed, the latter follows from the fact that $\mathbb F_q^n$ contains less than $q^{2n}$ non-trivial 4-APs and each one survives the random process with probability $p^4$. In particular, the expected number of elements of $P$ is considerably larger than the expected number of 4-APs. Therefore, with high probability both
\[
|P| \geq \frac{1}{1000}q^{2n/3}
\]
and
\[
|\{ \text{all non-trivial 4-APs in } P\}| \leq \frac{1}{2000}q^{2n/3}
\]
hold. We can then delete one element from each 4-AP and obtain a set $P'$ with size $\Omega(q^{2n/3})$ which has no nontrivial four-term progressions.

On the other hand, we can apply Theorem \ref{thm:CG'} with $t=\frac{1}{3(C_q-1)}$ and the above choice of $p$, as these values satisfy the required conditions provided that $n$ is sufficiently large. Therefore, with probability tending to $1$ as $n$ goes to infinity, the randomly constructed set $P$ satisfies
\[
f_3(P) \leq pq^{n\left(1-\frac{1}{3(C_q-1)} +2\beta\right)} \ll  q^{n\left(\frac{2}{3}-\frac{1}{3(C_q-2)}+2\beta\right)}.
\]

Now, for every positive integer $m$, $P'$ contains a three-term progression-free set of size $m$ only if $P$ also does. That is, $f_3(P') \leq f_3(P)$. Therefore,
\[
f_3(P') \leq f_3(P) \ll  q^{n\left(\frac{2}{3}-\frac{1}{3(C_q-2)} +2\beta\right )} \ll |P'|^{1-\frac{1}{2(C_q-2)}+3\beta}.
\]
This completes the proof.
\end{proof}

\smallskip


\section{Proof of Theorem \ref{thm:integers}}

\bigskip

We will prove the following more general result which involves the parameter $r_3(N)$.

\smallskip

\begin{proposition} \label{thm:integersgen}
Suppose that for all sufficiently large $N \in \mathbb N$ we have $r_3(N) \leq \frac{N}{h(N)}$ for some monotone increasing and invertible function $h: \mathbb [1, \infty) \rightarrow \mathbb [1, \infty )$. Suppose also that $h$ satisfies the following technical conditions:
\begin{itemize}
    \item For all $x \in \mathbb [1,\infty)$, $h(x) \leq x$.
    \item There exists an absolute constant $\gamma$ such that for all $N$ sufficiently large
      \begin{equation} \label{technical1}
    h\left (\frac{N^{1/5}}{1000}\right) \geq 4h(N^{\gamma}), \,\, \text{and}
    \end{equation}
    \begin{equation} \label{technical}
    N^{1/10} \geq [h(N^{\gamma})]^{3/2} [h^{-1}( 4h(N^{\gamma}))]^2.
    \end{equation}
\end{itemize}
Then for all $\alpha>0$ and for all $n$ sufficiently large (depending on $\alpha$), there exists a four-term progression-free set $A \subset \mathbb N$ with cardinality $n$ such that
\[
f_3(A) \ll \frac{n}{[h(n^{\frac{3}{2}\gamma})]^{1-\alpha}}.
\]

\end{proposition}

\begin{proof}[Proof of Theorem \ref{thm:integers} from Proposition \ref{thm:integersgen}]

Note that the rather complicated looking statement of Proposition \ref{thm:integersgen} does imply the upper bound from Theorem \ref{thm:integers}. Indeed, because of Bloom and Sisask's recent upper bound on $r_3(N)$ in \eqref{r3}, we can apply Proposition \ref{thm:integersgen} with 
\[
h(x)=\frac{1}{C}(\log x)^{1+c}, \,\,\,\,\,\,\,\,\, \alpha=\frac{c}{2(1+c)},
\] where $c$ is the constant given in \eqref{r3} and $C$ is the multiplicative constant hidden in the $\ll$ notation therein. One can check by direct calculation that $h$ does indeed satisfy the conditions of Theorem \ref{thm:integersgen}, with room to spare, if we take $\gamma=\frac{1}{100}$, and so
\[
f_3(A) \leq \frac{C'n}{(\log n)^{(1+c)(1-\alpha)}}=\frac{C'n}{(\log n)^{(1+\frac{c}{2})}}.
\]
This completes the proof of Theorem \ref{thm:integers} with $c'=c/2$.

\end{proof}


\bigskip

\begin{proof}[Proof of Proposition \ref{thm:integersgen}]

The proof is similar to that of Theorem \ref{thm:main}, although the calculations are more taxing. On the other hand, this proof is a little more straightforward, since we make just a single application of the container theorem. We remark that this approach with a single application was also possible in the proof of Theorem \ref{thm:main}, but the iterative approach gave a better quantitative result. However, the quantitative gains of the iterative approach seem to be negligible in the integer case.

Once again, we define a $3$-uniform hypergraph which encodes three term arithmetic progressions. This hypergraph $H$ has vertex set $[N]$, and three distinct elements of $[N]$ form an edge if they form an arithmetic progression.

Note that the average degree $d$ of this hypergraph is at least $N/9$, since there are at least $N^2/9$ edges. Indeed, if we take any two distinct integers $a,b \in [1,N/2]$ with $a<b$, there exists a third integer $c=2b-a \in [1,N]$ such that $\{a,b,c\}$ forms an arithmetic progression. This shows the existence of at least
\[
\binom{\lfloor \frac{N}{2} \rfloor}{2} > \frac{N^2}{9}
\]
non-trivial $3$-APs, where the latter inequality holds provided that $N$ is sufficiently large. Also, as in the proof of Theorem \ref{thm:main}, we have $\Delta_2 \leq 3$ and $\Delta_3=1$.

Fix
\[ 
\eta:= \frac{1}{h(N^{\gamma})},
\]
where $\gamma$ is the constant in the statement of Proposition \ref{thm:integersgen}. Define 
\[ 
\epsilon:= \frac{\eta}{\left(h^{-1}(\frac{4}{\eta}) \right)^4}, \,\,\,\,\,\,\, \tau:= \frac{100}{(N \epsilon)^{1/2}}.
\]

We would like to apply Theorem \ref{thm:container} with these parameters. In order to do this, we need to check that the conditions 
\begin{equation} \label{cond1}
    \tau < 1/(200 \cdot 3 \cdot 3!^2)= 1/21600
    \end{equation}
and    
    \begin{equation} \label{cond2}
   \Delta(H, \tau) \leq \frac{\epsilon}{72}
    \end{equation}
    hold. 
    
    
For \eqref{cond1} to hold, it would be enough to verify that
\begin{equation}
    \label{Nepsilon}
N \epsilon \geq 10^{12}.
\end{equation}
That is,
\begin{equation} \label{aim}
\frac{\eta}{\left(h^{-1}(\frac{4}{\eta}) \right)^4} \geq \frac{10^{12}}{N}.
\end{equation}

Because of the assumption that $h(x) \leq x$ for all $x \in \mathbb R^+$, it follows that  $h^{-1}(x) \geq x$ and in particular
\begin{equation} \label{messy}
    \frac{1}{x} \geq \frac{1}{h^{-1}(x)}.
\end{equation}
Applying \eqref{messy} with $x=\frac{4}{\eta}$, it follows that
\[
\frac{\eta}{\left(h^{-1}(\frac{4}{\eta}) \right)^4}= 4\frac{\frac{\eta}{4}}{\left(h^{-1}(\frac{4}{\eta}) \right)^4} \geq 4 \frac{1}{\left(h^{-1}(\frac{4}{\eta}) \right)^5},
\]
so that \eqref{aim} would hold as long as 
\[
\frac{1}{\left(h^{-1}(\frac{4}{\eta}) \right)^5} \geq \frac{10^{12}}{N}.
\]
Since $h$ is monotone increasing, this can be rearranged to give
\[\eta \geq \frac{4}{h\left( \frac{N^{1/5}}{10^{12/5}} \right)}.
\]
The latter inequality holds for our choice of $\eta$. Here we have used the condition \eqref{technical1} in the statement of the theorem. This implies that \eqref{aim} holds, and therefore so does \eqref{cond1}.


For \eqref{cond2} to hold, we need to verify that
\begin{equation} \label{eq:condition2}
 \frac{4\cdot 9 \cdot 3}{N \tau} + \frac{2\cdot 9}{N \tau^2}  \leq \frac{\epsilon}{72}.
\end{equation}
Since $\tau <1$, it will be sufficient to check that $\frac{126}{N\tau^2} \leq \frac{\epsilon}{72}$. By the earlier choice of $\tau$, this is equivalent to $100^2 \geq 72 \cdot 126$, which is indeed true.


Theorem \ref{thm:container} then gives a collection $\mathcal C$ of subsets of $[N]$ such that
\begin{itemize}
    \item $|\mathcal C| \leq 2^{c\tau N\log(\frac{1}{\tau}) \log (\frac{1}{\epsilon})} $  ,
    \item for all $A \in \mathcal C$, $e(H[A]) \leq \epsilon e(H)$,
    \item if $X \subseteq [N]$ is an independent set in $H$, then there is some $A \in \mathcal C$ such that $X \subseteq A$.
\end{itemize}

It follows from the second fact above and Theorem \ref{thm:varnavides} that $|A| \leq \eta N$ for all $A \in \mathcal C$. Note here that the condition of Theorem \ref{thm:varnavides} follows from condition \eqref{technical}.

Observe that, for $N$ sufficiently large
\[
\frac{1}{\tau}, \frac{1}{\epsilon} \leq N.
\]
The first of these inequalities follows from the fact that $\epsilon <\frac{1}{2}$, while the second is a consequence of \eqref{Nepsilon}. Using these two inequalities and the definition of $\tau$, gives the bound
\begin{equation} \label{Cbound}
|\mathcal C| \leq 2^{c'(\log N )^2 \left(\frac{N}{\epsilon}\right)^{1/2}}.
\end{equation}

Construct a subset $P \subset [N]$ by choosing elements independently at random with probability $p$. The expected number of elements in $P$ is $pN$. The expected number of four-term arithmetic progressions is at most $p^4N^2$. Therefore, if we choose $p=\frac{1}{100} N^{-1/3}$ then with high probability the number of elements will be much larger than the number of four-term arithmetic progressions. We can then delete one element from each $4$-AP and obtain a set $P'$ with size $\Theta(N^{2/3})$ which has no $4$-APs. Just as was the case in the proof of Theorem \ref{thm:main}, note here that $f_3(P') \leq f_3(P)$.

Now, we claim that it is unlikely that $H(P)$ contains an independent set of cardinality $m= N^{2/3}\eta^{1-\alpha}$. Indeed, note that
$$\mathbb P [H(P)\text{ contains an independent set of size }m] \leq |\mathcal C| \binom{N\eta}{m} p^m,$$
whereas, by using the bound on $|\mathcal C|$ from \eqref{Cbound} together with standard binomial coefficient estimates, we also have that
$$|\mathcal C| \binom{N\eta}{m} p^m \leq 2^{c'(\log N)^2  \left(\frac{N}{\epsilon}\right)^{1/2}}  \left( \frac{eN\eta}{mN^{1/3}} \right)^m = 2^{c'(\log N)^2  \left(\frac{N}{\epsilon}\right)^{1/2}}  (e \eta^{\alpha})^m .
$$
With the choices we have made for $\eta$ and $m$, it follows that the bound
\[
c'(\log N)^2  \left(\frac{N}{\epsilon}\right)^{1/2} \leq m
\]
holds for $N$ sufficiently large. This is the point at which we have used the technical condition \eqref{technical} in the statement of Proposition \ref{thm:integersgen}. Therefore, the probability that $H(P)$ contains an independent set of size $m$ is less than $\left(2e \eta^{\alpha}\right)^m$, which becomes arbitrarily small as $N$ gets arbitrarily large. So, with high probability, $H(P)$ does not contain any independent sets of this size, and thus neither does the induced subhypergraph $H(P')$.

It follows that there exists a $4$-AP free set $P'$ of size $\Theta(N^{2/3})$ with the property that all of its subsets of size at least $N^{2/3}\eta^{1-\alpha}$ contain a $3$-AP. That is,
\[
f(|P'|) \ll |P'| \eta^{1-\alpha} \approx \frac{|P'|}{\left(h(|P'|^{\frac{3}{2}\gamma})\right)^{1-\alpha}}.
\]
This completes the proof of Proposition \ref{thm:integersgen}. 
\end{proof}

\smallskip

\section{Sets with small energy but rich in progressions}

\bigskip

In this section, we discuss another application of Theorem \ref{thm:CG'}, in connection with a different type of generalization of Roth's theorem, first observed by Sanders \cite{Sanders}. 

\smallskip

\begin{theorem}
\label{Sanders}
Let $\delta > 0$ and suppose that $A \subset \mathbb{Z}$ has at least $\delta|A|^{3}$ additive quadruples. Then, there exist absolute constants $c,C>0$ such that $A$ contains at least $\exp(-C\delta^{-c}) \cdot |A|^{2}$ three-term arithmetic progressions.
\end{theorem}

\smallskip

Here an additive quadruple means a solution to $a+b=c+d$ with all $a,b,c,d$ in $A$. The number of such quadruples is usually denoted by $E(A)$ and called the {\it{additive energy}} of $A$. Theorem \ref{Sanders} says that sets with large energy have many three-term arithmetic progressions. This follows from the Balog-Szemer\'edi-Gowers theorem (see \cite{Gowers} or \cite{tv}) and the fact that sets with small sumsets have many three-term arithmetic progressions, a consequence of Roth's theorem. Results like the latter hold in general abelian groups $G$ and quantitative versions were also studied by Henriot in \cite{Henriot}. For our purposes, the groups of interest are $G = \mathbb{Z}$ and $G = \mathbb{F}_{q}^n$, so we begin by recording an improvement (and generalisation) of a theorem of Henriot \cite[Theorem 6]{Henriot}, which may be of independent interest, and which is meant to illustrate a phenomenon similar to the one described by Theorem \ref{Sanders} (with better quantitative bounds).

\smallskip

\begin{theorem}
\label{Henriot}
Let $A \subset \mathbb{F}_{q}^n$ be such that $|A+A| \leq K|A|$ for some $K > 0$. Then, $A$ contains at least $(qK^{4})^{2-C_{q}} \cdot |A|^{2}$ three-term arithmetic progressions.
\end{theorem}

\smallskip

\begin{proof} For the reader's convenience, we recall that for any two commutative groups $G_{1}$, $G_{2}$ two sets $S \subset G_{1}$ and $T \subset G_{2}$ are said to be Freiman $s$-isomorphic if there exists a one to one map $\phi : S \to T$ such that for every $x_{1},\ldots,x_{s},y_{1},\ldots,y_{s}$ in $S$ (not necessarily distinct) the equation
$$x_{1}+\ldots+x_{s} = y_{1} + \ldots + y_{s}$$
holds if and only if
$$\phi(x_{1})+\ldots+\phi(x_{s}) = \phi(y_{1}) + \ldots + \phi(y_{s}).$$

Let $K = |A+A|/|A|$. By a finite field version of the so-called Freiman-Ruzsa modelling lemma (see for instance \cite[Lemma 5.6]{Sisask} for more details), $A$ is Freiman $2$-isomorphic to a subset of $G = \mathbb{F}_{q}^{m}$, where $|G| \leq q \cdot K^{4} |A|$. We identify this subset with $A$ since the Freiman $2$-isomorphisms preserves three-term progressions. By Corollary \ref{thm:supersaturation} applied inside $G$, it follows that $A$ contains at least $|A|^{2}(qK^{4})^{2-C_{q}}$ three-term arithmetic progressions, as claimed.
\end{proof}

Theorem \ref{Henriot}, combined with the Balog-Szemer\'edi-Gowers theorem, shows that subsets $A \subset \mathbb{F}_{q}^{n}$ must have many three-term progressions even if $E(A) \gg |A|^{3-\epsilon}$ for some $\epsilon > 0$ (which depends on $q$). A natural question now seems to be: if $A$ has large additive energy, does it also mean that $A$ must have nontrivial three-term progressions in all large subsets? A naive view is that Theorem \ref{Sanders} and Theorem \ref{Henriot} suggest that the answer could be yes. However, a simple counterexample already points towards the contrary: consider a set of $A$ where half of the elements form an additively structured set (like an arithmetic progression), while the other half consists of random elements. It is easy to check that $E(A) \gg |A|^{3}$ because the additively structured part has large energy, while there is no reason why the random part should contain any non-trivial three-term progressions. 

We will push this observation one step further and show next that for sets in $\mathbb{F}_{q}^{n}$ or $\mathbb{Z}$ the property of ``having nontrivial three-term progressions in all large subsets" is in fact entirely uncorrelated with the property of ``having large additive energy". 

\smallskip

\begin{theorem}
For all $\epsilon>0$ and any prime power $q$ there exists $\delta(\epsilon,q):=\delta>0$ and $n_0=n_0(\epsilon,q)$ such the following statement holds. For all $n \geq n_0$ there exists a set $A \subset \mathbb F_q^n$ with 
\[E(A) \leq |A|^{2+\epsilon}\]
and 
\[f_3(A) \ll |A|^{1-\delta}.\]
\end{theorem}

\smallskip

In other words, not only is it true that sets with large additive energy may have large subsets with no proper three-term progressions, but there also exist sets with low energy with the property that all their large subsets contain nontrivial three-term progressions. The proof uses again Theorem \ref{thm:CG'} and is similar to the proof of Theorem \ref{thm:main}.

\smallskip

\begin{proof}
Construct a subset $P \subset \mathbb F_q^n$ by choosing elements independently at random with probability $p= q^{n\left( -\frac{1}{2} + \frac{\epsilon}{4-2\epsilon}\right)}$. The expected number of elements in $P$ is $pq^n=q^{n\left (\frac{1}{2}+\frac{\epsilon}{4-2\epsilon} \right)}$.

The expected size of $E(P)$ is $p^4q^{3n}=q^{n\left ( 1+\frac{4\epsilon}{4-2\epsilon} \right )}$. Indeed, this follows from the fact that there are $q^{3n}$ solutions to the equation
\[
a+b=c+d,\,\,\,\,\,\,\, a,b,c,d \in \mathbb F_q^n
\]
 and each solution survives the random process with probability $p^4$. Therefore, with high probability both
\[
|P| \geq \frac{1}{100}q^{n\left (\frac{1}{2}+\frac{\epsilon}{4-2\epsilon} \right)}
\]
and
\[
E(P) \leq 100 q^{n\left ( 1+\frac{4\epsilon}{4-2\epsilon} \right )}
\]
hold. In particular, with high probability,
\[ 
E(P) \ll |P|^{2+\epsilon}.
\]

On the other hand, we can apply Theorem \ref{thm:CG'} with 
\[t=\frac{2\epsilon}{(4-2\epsilon)(C_q-1)},\,\,\,\,\,\,\, \beta= \frac{t}{4}.
\] 
The above choice of $p$ is admissible for these choices of $t$ and $\beta$. Therefore, with probability tending to $1$ as $n$ goes to infinity, the randomly constructed set $P$ satisfies
\[
f_3(P) \ll pq^{n(1-t+2\beta)} = pq^{n(1-\frac{t}{2})}=q^{n\left(\frac{1}{2}+\frac{\epsilon(C_q-2)}{(4-2\epsilon)(C_q-1)}\right)} = |P|^{1-\delta},
\]
where
\[
\delta=\frac{\epsilon}{2(C_q-1)}.
\]
This completes the proof.

\end{proof}

\smallskip

A similar statement can be established in the integer case, which we state without proof as follows.

\smallskip

\begin{theorem} \label{lowenergy}
For all $ \epsilon>0$ there exists $c>0$ and a set $A \subset \mathbb N$ such that
\[
E(A)\ll |A|^{2+\epsilon}
\]
and
\[
f_3(A) \ll_{\epsilon} \frac{|A|}{(\log |A|)^{1+c}}.
\]

\end{theorem}

\bigskip

We end this section with an epilogue on the optimality of Theorem \ref{lowenergy}. For this purpose, we recall a theorem of Koml\'{o}s, Sulyok and Szemer\'{e}di \cite{KSS}.

\begin{theorem} \label{thm:KSS} There is an absolute constant $c>0$ such that for any sufficiently large set $A \subset \mathbb Z$,
\[
f_3(A) \geq c \cdot f_3(\{1,\dots,|A|\}) = c\cdot r_3(|A|).
\]
\end{theorem}
Essentially, Theorem \ref{thm:KSS} tells us that $f_3(A)$ is minimal as a function of $|A|$ when $A$ is an interval.\footnote{In fact, \cite{KSS} gives much more general information about systems of linear equations, but the version stated as Theorem \ref{thm:KSS} corresponds to the case we are interested in in this paper.} Combining this with Elkin's theorem
\[
r_3(N) \gg \frac{\log^{1/4}{N}}{2^{2\sqrt{2} \sqrt{\log N}}} \cdot N,
\]
it follows that \textit{every} sufficiently large set $A \subset \mathbb Z$ contains a three-term progression free subset of cardinality at least
\begin{equation} \label{KSSelkin}
\Omega \left (\frac{\log^{1/4}{|A|}}{2^{2\sqrt{2} \sqrt{\log |A|}}} \cdot |A|\right).
\end{equation}
So Theorem \ref{lowenergy} is as close to optimal as the upper bound for $r_{3}(N)$ in \eqref{r3} is close to optimal. Note however that in this observation we have not used the additional hypothesis that $A$ has low additive energy. The next natural question therefore seems to be: is it possible to get a significantly better bound than 
\begin{equation}
\label{flower}
f_3(A) \gg \frac{\log^{1/4}{|A|}}{2^{2\sqrt{2} \sqrt{\log |A|}}} \cdot |A|
\end{equation}
for {\textit{all}} sets $A \subset \mathbb{Z}$ satisfying $E(A) \ll |A|^{2+\epsilon}$ for some (or even all) $0< \epsilon <1$? This time, the answer turns out to be (a modest) {\it{yes}}.

\begin{theorem} \label{Elkinwin}
Let $0 < \epsilon < 1$ and let $A \subset \mathbb{Z}$ be such that $E(A) \ll |A|^{2+\epsilon}$. Then,
$$f_{3}(A) \gg \frac{\log^{1/4}{|A|}}{2^{2\sqrt{(1+\epsilon)\log {N}}}} \cdot N.$$
\end{theorem}

In particular, {\it{all}} sets with $E(A) \ll |A|^{2+\epsilon}$ for all $\epsilon > 0$ have slightly larger 3AP-free sets than we know $\left\{1,\ldots,N\right\}$ must have. Our argument follows closely the alternative proof of Elkin's bound due to Green and Wolf from \cite{GW}, which can be easily modified to start with a general set of $N$ integers instead of the interval $\left\{1,\ldots,N\right\}$. The main observation is that for a set $A$ with $E(A) \ll |A|^{2+\epsilon}$ for some $0 < \epsilon < 1$, we have a power saving on the total number $T(A)$ of three-term progressions with elements in $A$. Indeed, for each element $s \in A+A$, let $r_{A+A}(s)$ denote the number of pairs $(x,y) \in A \times A$ such that $x+y=s$. For each $b \in A$, note that $r_{A+A}(2b)$ represents the number of three-term progressions centered at $b$. By Cauchy-Schwarz,
$$T(A)^{2} = \left(\sum_{b \in A} r_{A+A}(2b)\right)^{2} \leq |A| \left(\sum_{b \in A} r_{A+A}^{2}(2b)\right).$$
Since
$$\sum_{b \in A} r_{A+A}^{2}(2b)\leq \sum_{s \in A+A} r_{A+A}^{2}(s) = E(A),$$
it follows that $T(A)^{2} \leq |A| \cdot E(A) \ll |A|^{3+\epsilon}$, i.e. $T(A) \ll |A|^{(3+\epsilon)/2}$. Theorem \ref{Elkinwin} will then follow from the following more general result.

\smallskip

\begin{proposition} \label{prop:t(A)}
Let $A \subset \mathbb Z$ be a set of size $N$ such that the number of three-term progressions satisfies $T(A)= N^2 / t(A)$. Then $A$ contains a three-term progression free subset $A'$ such that
\[
|A'| \gg N \cdot \frac{\left[\log \left( \frac{N}{t(A)} \right) \right]^{1/4} }{2^{2 \sqrt{2\log_2 \left( \frac{N}{t(A)} \right)}}}.
\]
\end{proposition}

\smallskip

{\textit{Proof.}} Let $N$ be a sufficiently large positive integer and let $A$ be some four-term progression free set of size $N$. Let $d$ be a positive integer to be precisely determined later (but which we shall think of as sufficiently large for the time being), and let $\mathbb{T}^{d} = \mathbb{R}^{d}/\mathbb{Z}^{d}$ denote the $d$-dimensional torus. For each $\theta,\alpha \in \mathbb{T}^{d}$, let $\Psi_{\theta,\alpha} :\ A \to \mathbb{T}^{d}$ be the map defined by
\begin{equation}
    \label{psi}
    n \mapsto \theta n + \alpha \mod 1.
\end{equation}
For a fixed $n$ integer, as we let $\theta$, $\alpha$ vary uniformly and independently over $\mathbb{T}^{d}$, the image $\Psi_{\theta,\alpha}$ is uniformly distributed on the $d$-dimensional torus. Moreover, it is also true that the pair of points
\begin{equation} \label{UD}
\left(\Psi_{\theta,\alpha}(n),\Psi_{\theta,\alpha}(n')\right)\ \text{is uniformly distributed on}\ \mathbb{T}^{d} \times \mathbb{T}^{d}
\end{equation}
as $\theta,\alpha$ vary uniformly and independently over $\mathbb{T}^{d}$, provided that integers $n$ and $n'$ are distinct. Indeed,
$$\int e^{2\pi i(k \cdot (\theta n + \alpha) + k' \cdot (\theta n'+\alpha))} d\theta d\alpha = 0$$
unless $k+k'=kn+k'n'=0$, which is however impossible if $n$ and $n'$ are distinct. Since the exponentials $e^{2\pi i (kx+k'x')}$ are dense in $L^{2}(\mathbb{T}^{d} \times \mathbb{T}^{d})$, the claim checks out.

 Fix $\delta$ to be a positive constant which we will declare later. We identify the $d$-dimensional torus $\mathbb{T}^{d}$ with $[0,1)^{d}$, and for each $r \leq \frac{1}{2} \sqrt{d}$, we define the annulus
$$S(r):= \left\{x \in [0,1/2]^{d}:r - \delta \leq \|x\|_{2} \leq r\right\}.$$
Like in Lemma 2.2 from \cite{GW}, out of all of the possible values of $r$, we choose the one for which $S:=S(r)$ satisfies 
\begin{equation}
    \label{annulus}
\operatorname{vol}(S(r)) \geq c\delta 2^{-d},
\end{equation}
for some absolute constant $c$. 

Finally, for each $\theta,\alpha$ chosen uniformly and independently at random on $\mathbb{T}^{d}$, we let $A_{\theta,\alpha}$ be the subset of $A$ defined by
$$A_{\theta,\alpha}:=\left\{n \in A: \Psi_{\theta,\alpha}(n) \in S\right\},$$
where $\Psi_{\theta,\alpha}$ is the map from \eqref{psi}. By \eqref{UD}, the expected size of $A_{\theta,\alpha}$ satisfies
\begin{equation}
    \label{size}
    \mathbb{E}_{\theta,\alpha}|A_{\theta,\alpha}| = N \cdot \operatorname{vol}(S),
\end{equation}
while the expected number $T(A_{\theta,\alpha})$ of three term progressions in $A_{\theta,\alpha}$ is
\begin{equation}
    \label{3APs}
    \mathbb{E}_{\theta,\alpha}T(A_{\theta,\alpha}) = T(A) \cdot \operatorname{vol}(\Upsilon).
\end{equation}
Here $\Upsilon$ represents the set points $(x,y) \in \mathbb{T}^{d} \times \mathbb{T}^{d}$ so that $x-y$, $x$ and $x+y$ all lie in $S$. 

We can upper bound the volume of $\Upsilon$ as follows. By the parallelogram law
$$2\|x\|^{2}+2\|y\|_{2}^{2} = \|x+y\|_{2}^{2} + \|x-y\|_{2}^{2},$$
so
$$\|y\|_{2} \leq \sqrt{r^{2}-(r-\delta)^{2}} \leq \sqrt{2 \delta r}.$$
If $V_{d}$ denotes the volume of the unit ball in $\mathbb{R}^{d}$, then this implies
$$\operatorname{vol}(\Upsilon) \leq \operatorname{vol}(S) \cdot (\sqrt{2\delta r})^{d} V_{d}.$$
On the other hand, we have the estimate
$$V_{d} \ll 10^{d} d^{-d/2};$$
therefore 
$$\operatorname{vol}(\Upsilon) \leq \operatorname{vol}(S) \cdot (\sqrt{2\delta r})^{d} 10^{d} d^{-d/2} \leq \operatorname{vol}(S) \cdot 10^{d} \left(\frac{\delta}{\sqrt{d}}\right)^{d/2}.$$
By \eqref{3APs}, this estimate implies
$$\mathbb{E}_{\theta,\alpha}T(A_{\theta,\alpha}) = T(A) \cdot \operatorname{vol}(\Upsilon) \leq C \frac{N^{2}}{t(A)} \cdot \operatorname{vol}(S) \cdot 10^{d} \left(\frac{\delta}{\sqrt{d}}\right)^{d/2},$$
for an absolute constant $C>0$. Now, if we choose $\delta$ and $d$ so that
\begin{equation}
    \label{choice}
    10^{d} \left(\frac{\delta}{\sqrt{d}}\right)^{d/2} \leq \frac{1}{3C} \cdot \frac{t(A)}{N},
\end{equation}
then by \eqref{size}
$$\mathbb{E}_{\theta,\alpha}T(A_{\theta,\alpha}) = T(A) \cdot \operatorname{vol}(\Upsilon) \leq \frac{1}{3} \cdot N \cdot \operatorname{vol}(S) = \frac{1}{3} \cdot \mathbb{E}_{\theta,\alpha}|A_{\theta,\alpha}|.$$
Consequently, by  deleting one element from each progression appearing in $A_{\theta,\alpha}$, the remaining subset $A'_{\theta,\alpha} \subset A_{\theta,\alpha} \subset A$ is three-term progression-free. Moreover, $A'_{\theta,\alpha}$ has expected size
$$\mathbb{E}_{\theta,\alpha}|A'_{\theta,\alpha}| \geq \frac{2}{3}\cdot \mathbb{E}_{\theta,\alpha}|A_{\theta,\alpha}| \geq \frac{2}{3} \cdot N \cdot \operatorname{vol}(S) \gg N \delta 2^{-d},$$
where the last inequality follows from \eqref{annulus}. In particular, there exists a specific choice of $\theta,\alpha \in \mathbb{T}^{d}$ so that $A' := A'_{\theta,\alpha}$ is a three-term progression free subset of $A$ for which
$$|A'| \gg N \delta 2^{-d}.$$
Finally, take
$$\delta:= C' \sqrt{d} \cdot \left(\frac{t(A)}{N}\right)^{2/d}$$
for some absolute constant $C'>0$ so that \eqref{choice} is achieved. For this choice, we have
$$|A'| \gg N \delta 2^{-d} \gg \sqrt{d} \cdot t(A)^{2/d} N^{1-2/d} \cdot 2^{-d}.$$
Set
\[
d= \left \lceil \sqrt{2 \log_2 \left( \frac{N}{t(A)} \right )} \right \rceil.
\]
It then follows that
\[
|A'| \gg N \cdot \frac{\left[\log_2 \left( \frac{N}{t(A)} \right) \right]^{1/4} }{2^{2 \sqrt{2\log_2 \left( \frac{N}{t(A)} \right)}}}.
\]
This concludes the proof of Proposition \ref{prop:t(A)} and thus that of Theorem \ref{Elkinwin} (one can check that taking $t(A) = \Theta(N^{(1-\epsilon)/2})$ in Proposition \ref{prop:t(A)} yields the bound from Theorem \ref{Elkinwin}, as claimed).

\bigskip

\section{Concluding remarks}
In this last section, we would like to end with a few more words on the upper bound from Theorem \ref{thm:integers}. In light of Theorem \ref{thm:KSS}, this is as good in some sense as the upper bound for $r_{3}(N)$ from \eqref{r3} but, like in the second part of Section 6, one can then similarly ask whether it is possible to improve on
\begin{equation}
\label{flower4}
f_3(A) \gg \frac{\log^{1/4}{|A|}}{2^{2\sqrt{2} \sqrt{\log |A|}}} \cdot |A|
\end{equation}
for sets $A$ without nontrivial four-term progressions. In Theorem \ref{thm:integers}, the 4-AP-free set $A$ we constructed with 
$$ f_3(A) \ll \frac{1}{(\log N)^{1+\epsilon}} \cdot N$$
also happened to satisfy the property that $T(A) = \Theta(|A|^{3/2})$, so by Proposition \ref{prop:t(A)} it also has larger three-term progression free sets than we know $\left\{1,\ldots,N\right\}$ must have, namely
$$f_{3}(A) \gg \frac{\log^{1/4}{N}}{2^{2\sqrt{\log {N}}}} \cdot N.$$
In \cite{GR}, Gyarmati and Ruzsa also improved on \eqref{flower4} when $A = \left\{1,2^2,\ldots,N^{2}\right\}$ by more number theoretic means that are quite specific to perfect squares. However, is it possible to get a bound better bound than \eqref{flower4} for {\textit{all}} $4$-AP free sets $A$? A construction of Fox and the first author from \cite{FP} shows that four-term progression free sets of size $N$ may sometimes contain $\gg N^{2}/2^{3(\log N)^{1/3}}$ three-term progressions, so our Proposition \ref{prop:t(A)} doesn't yield any asymptotic gain over the Elkin lower bound in general. It would be interesting if other methods would be able to provide such a result.

\bigskip

\providecommand{\bysame}{\leavevmode\hbox to3em{\hrulefill}\thinspace}

\end{document}